\documentclass{amsart}
\date{\today}
\usepackage{amsmath,amsthm,amssymb,mathrsfs,enumerate,comment}

\title[A spectral equivalent condition of the
$P$-polynomial property]
{A spectral equivalent condition of the $P$-polynomial property
for association schemes}

\author[H.~Kurihara and H.~Nozaki ]
{Hirotake Kurihara, Hiroshi Nozaki}

\address{Department of Integrated Arts and Science,
Kitakyushu National College of Technology,
5-20-1 Shii, Kokuraminamiku, Kitakyushu, Fukuoka, 802-0985,
Japan.
}
\email{kurihara@kct.ac.jp}

\address{Department of Mathematics
Aichi University of Education
1 Hirosawa, Igaya-cho, Kariya-city, Aichi, 448-8542, 
Japan.
}
\email{hnozaki@auecc.aichi-edu.ac.jp}


\numberwithin{equation}{section}
\newtheorem{thm}{Theorem}[section]
\newtheorem*{thm*}{Theorem}
\newtheorem{lem}[thm]{Lemma}

\theoremstyle{definition}

\theoremstyle{remark}
\newtheorem{rem}[thm]{Remark}

\newcommand{\R}{\mathbb{R}}
\newcommand{\C}{\mathbb{C}}

\def\set#1#2{\{#1\,|\,#2\}}
\newcommand{\trans}[1]{{}^t\hspace{-0.4ex}#1}

\newcommand{\rank}{\mathrm{rank}\,}

\newcommand{\aprod}[2]{
\langle #1,#2 \rangle
}

\begin{document}
\begin{abstract}
We give two equivalent conditions of the $P$-polynomial property of a symmetric association scheme. 
The first equivalent condition shows that the $P$-polynomial property is determined only by the 
first and second eigenmatrices of the symmetric association scheme. The second equivalent condition is another expression of 
the first using predistance polynomials.  
\end{abstract}

\maketitle

\section{Introduction}

An association scheme
is a finite set
with  higher regularity of relations among
the elements of the set.
As interesting objects,
$P$- or $Q$-polynomial association schemes are
closely related to
the theory of orthogonal polynomials or linear programming~\cite{Bannai1984aci, Delsarte1973aat},
and are main topics in algebraic combinatorics.
Kurihara and Nozaki~\cite{Kurihara2012coq} found
a simple equivalent condition of the $Q$-polynomial property
using the first and second eigenmatrices of the association scheme.
In this paper,
we give a ``dual'' version of the above theorem
for the $P$-polynomial property
using the first eigenmatrix whose entries are
the eigenvalues of the adjacency matrices.
Namely the following is a spectral equivalent condition of the $P$-polynomial property
for association schemes.
\begin{thm}
\label{thm:main_thm}
Let $\mathfrak{X}=(X,\{R_i\}^d_{i=0})$ be a symmetric
association scheme of class $d$.
Suppose that the entries $\{\theta_j\}^d_{j=0}$
of the first column of the first eigenmatrix of $\mathfrak{X}$
are mutually distinct.
Then the following are equivalent:
\begin{enumerate}[{$($}1{$)$}]
\item $\mathfrak{X}$ is a $P$-polynomial association scheme
with respect to the first adjacency matrix $A_1$.
\item There exists $l\in \{0,1,\ldots ,d\}$ such that
for each $h \in \{1,2,\ldots ,d\}$,
\begin{equation}
\label{eq:important}
\prod^{d}_{\substack{j=1\\
j \neq h}}
\frac{\theta_0-\theta_j}{\theta_h-\theta_j}= -Q_h(l), 
\end{equation}
where $Q_h(l)$ is the $(l,h)$-entry of the second eigenmatrix of $\mathfrak{X}$.

\end{enumerate}
Moreover if $($2$)$ holds, then $A_l$ is the $d$-th matrix with respect to the 
resulting polynomial ordering.
\end{thm}

After we obtained Theorem~\ref{thm:main_thm},
Nomura and Terwilliger~\cite{Nomura2011tmw}
gave its linear algebraic generalization.
In this paper,
we give a simple proof of this result
based on Kurihara and Nozaki~\cite{Kurihara2012coq}.
Moreover we investigate a relation between predistance polynomials
and this result.

$P$-polynomial association schemes
are identified with distance-regular graphs.
The spectrum of a graph can give several properties of the graph.
For example,
we can know regularity, bipartiteness, the number of the connected components,
and so on~\cite{Cvetkovi'c1997eog}.
Nevertheless  distance-regularity
is not determined only by the spectrum of a graph.
Fiol and Garriga~\cite{Fiol1997fla} gave a characterization of
distance-regularity using excesses and predistance polynomials determined by
the spectra of finite graphs, 
and this result is called a spectral excess theorem.
Here the excess of a vertex is the number of the vertices at distance $D$ from it, where $D$ is the diameter of the graph.
After this work other descriptions of the spectral excess theorem
were obtained by van Dam~\cite{Dam2008set}, 
and Fiol, Gago and Garriga~\cite{Fiol2010spo}.

The $P$-polynomial property demands a suitable ordering
of adjacency matrices.
On the other hand, Theorem~\ref{thm:main_thm} implies that
the $P$-polynomial property is determined
only by the spectrum of the first and the last matrices
with respect to the 
resulting polynomial ordering.
The last matrix has the information of the excess of the graph
$(X,R_1)$.
For a connected regular graph,
the value corresponding to the left-hand side in \eqref{eq:important}
can be expressed by using the predistance polynomials.
We give another expression of Theorem~\ref{thm:main_thm}
by predistance polynomials of the graph $(X,R_1)$.

\begin{thm}
\label{thm:predistance-main}
Let $\mathfrak{X}=(X,\{R_i\}^d_{i=0})$
be a symmetric association scheme of class $d$.
Suppose that the entries $\{\theta_j\}^d_{j=0}$
of the first column of the first eigenmatrix of $\mathfrak{X}$
are mutually distinct.
Let $\{p_i\}^d_{i=0}$ be the predistance polynomials
of the graph $(X,R_1)$.
Then the following are equivalent:
\begin{enumerate}[{$($}1{$)$}]
\item $\mathfrak{X}$ is a $P$-polynomial association scheme
with respect to $A_1$.
\item There exists $l\in \{0,1,\ldots ,d\}$ such that
for each $h\in \{0,1,\ldots ,d\}$,
$p_d(\theta_h)=P_l(h)$ holds,
where $P_l(h)$ is the $(h,l)$-entry of the first eigenmatrix of
$\mathfrak{X}$.
\end{enumerate}
\end{thm}

\section{Preliminaries}
\label{sec:2}
\subsection{Association schemes}
\label{sec:as}

We begin with a review of basic definitions concerning
association schemes.
The reader is referred to Bannai and Ito~\cite{Bannai1984aci} for 
background materials.

A {\it symmetric association scheme
$\mathfrak{X}=(X,\{R_i\}^{d}_{i=0})$ of class $d$}
consists of a finite set $X$ and a set $\{R_i\}^{d}_{i=0}$
of binary relations on $X$ satisfying
\begin{enumerate}
\item $R_0=\set{(x,x)}{x\in X}$,
\item $\{R_i\}^{d}_{i=0}$ is a partition of $X\times X$,
\item $\trans{R_i}=R_i$ for each $i\in \{0,1,\ldots ,d\}$,
where $\trans{R_i}=\set{(y,x)}{(x,y)\in R_i}$,
\item the numbers $|\set{z\in X}{\text{$(x,z)\in R_i$ and $(z,y)\in
R_j$}}|$ are constant whenever $(x,y)\in R_k$, for each $i,j,k\in
\{0,1,\ldots ,d\}$.
\end{enumerate}
Let $k_i$ denote the valency of the regular graph $(X,R_i)$.
Let $M_X(\mathbb{C})$ denote the algebra of matrices over
the complex field $\mathbb{C}$ with rows and columns indexed by $X$.
The $i$-th {\it adjacency matrix} $A_i$ in $M_X(\mathbb{C})$ of
$\mathfrak{X}$ is defined by
\[A_i(x,y)=\begin{cases}
1 & \text{if $(x,y)\in R_i$,}\\
0 & \text{otherwise.}
\end{cases}\]
The vector space
$\mathfrak{A}$
spanned by $\{A_i\}^d_{i=0}$
over $\C$ forms a commutative algebra,
and is called the {\it Bose--Mesner algebra} of
$\mathfrak{X}$.
It is well known that $\mathfrak{A}$ is semi-simple
\cite[Section 2.3]{Bannai1984aci},
hence $\mathfrak{A}$ has the primitive idempotents
$E_0, E_1,\ldots ,E_d$.
We call $m_j:=\rank E_j$ the \emph{multiplicities} of $\mathfrak{X}$.
The {\it first eigenmatrix} $P=(P_i(j))^d_{j,i=0}$
and the {\it second eigenmatrix} $Q=(Q_j(i))^d_{i,j=0}$
of $\mathfrak{X}$ are defined by 
\[A_i=\sum^{d}_{j=0}P_i(j)E_j\ 
\text{and}\ 
E_j=\frac{1}{|X|}\sum^{d}_{i=0}Q_j(i)A_i,\]
respectively.
Note that the pair of $\{P_i(j)\}^{d}_{j=0}$ and $\{m_j\}^{d}_{j=0}$
corresponds to the spectrum of the graph $(X,R_i)$.
In particular, we use the notation $\theta_j=P_1(j)$
for $0\le j\le d$.

A symmetric association scheme is called a
\emph{$P$-polynomial scheme}
(or a \emph{metric scheme})
with respect to the ordering $\{A_i \}_{i=0}^d$
(or $\{R_i \}_{i=0}^d$),
if for each $i \in \{0,1,\ldots, d \}$, 
there exists a polynomial $v_i$ of degree $i$,
such that $A_i=v_i(A_1)$.
Moreover, a symmetric association scheme is called
a $P$-polynomial scheme with respect to $A_1$
(or $R_1$)
if the symmetric association scheme
has the $P$-polynomial property
with respect to some ordering
$A_0,A_1, A_{\xi_2},A_{\xi_3},\ldots ,A_{\xi_d}$.
Dually, a symmetric association scheme is called
a \emph{$Q$-polynomial scheme}
(or a \emph{cometric scheme}) 
with respect to the ordering $\{E_j \}_{j=0}^d$,
if for each $j \in \{0,1,\ldots, d \}$, 
there exists a polynomial $v_j^{\ast}$ of degree $j$, 
such that $|X| E_j=v_j^{\ast}(|X| E_1)$,
where the multiplication is the entrywise product.
\begin{rem}
\label{rem:p-poly-drg}
For a $P$-polynomial scheme $(X,\{R_i\}^d_{i=0})$ with respect to $R_1$,
the graph $(X,R_1)$ is distance-regular.
Conversely,
for a distance-regular graph $(X,R)$ with diameter $d$,
letting $R_i$ be the collection $(x,y)\in X\times X$ such that the distance between $x$ and $y$ is $i$
for $i\in\{0,1,\ldots ,d\}$,
we have a $P$-polynomial scheme $(X,\{R_i\}^d_{i=0})$ with respect to the ordering $\{R_i \}_{i=0}^d$~\cite{Bannai1984aci}.
\end{rem}

\begin{rem}
Kurihara and Nozaki~\cite{Kurihara2012coq} found
a simple equivalent condition of the $Q$-polynomial property
using the character table of the association scheme.
The detail of the equivalent condition is as follows.
For  a symmetric
association scheme $\mathfrak{X}=(X,\{R_i\}^d_{i=0})$
such that $\{Q_1(i)\}^d_{i=0}$ are mutually distinct,
$\mathfrak{X}$ is a $Q$-polynomial association scheme
with respect to $E_1$ if and only if
there exists $l\in \{0,1,\ldots ,d\}$ such that
for each $h \in \{1,2,\ldots ,d\}$,
\[
\prod^{d}_{\substack{i=1\\
i \neq h}}
\frac{Q_1(0)-Q_1(i)}{Q_1(h)-Q_1(i)}= -P_h(l). 
\]
\end{rem}

\subsection{The predistance polynomials of graphs}

As we see in Remark~\ref{rem:p-poly-drg},
a distance-regular graph $\Gamma$ has the structure of a $P$-polynomial scheme
and so has polynomials $\{v_i\}$.
In this section, we give ``predistance polynomials'' for any regular connected graph,
and we consider relations between predistance polynomials and distance-regular graphs.

All graphs in this paper are assumed to be finite.
Let $\Gamma =(X,R)$ be a connected regular graph
with the adjacency matrix $A$ and the spectrum
$\{\theta_0 ^{m_0},\theta_1 ^{m_1},\ldots ,\theta_d ^{m_d}\}$,
where $\theta_0>\theta_1>\cdots >\theta_d$.
Let $Z(t):=\prod^d_{j=0}(t-\theta_j)\in \R[t]$.
We define
the inner product
on $\R[t]/(Z)$ by
\[
\aprod{p}{q}
=
\frac{1}{|X|}
\sum^{d}_{j=0}m_j p(\theta_j)q(\theta_j)
\]
for $p,q\in \R[t]/(Z)$.
The \emph{predistance polynomials} $p_0,p_1,\ldots ,p_d$ of $\Gamma$
are the unique polynomials
satisfying $\deg p_i=i$
and for $i,j \in \{0,1,\ldots ,d\}$,
$\aprod{p_i}{p_j}$ is $p_i(\theta_0)$
if $i=j$ and $0$ otherwise.
It is well known that $p_i(\theta_0)>0$
for any $i \in \{0,1,\ldots ,d\}$ \cite{Dam2008set}.

For a graph of diameter $D$,
the \emph{excess $\gamma _x$ of a vertex $x$}
is the number of the vertices at distance $D$ from $x$.
An application of predistance polynomials and excesses
is shown in the following theorem.
Namely, distance-regularity is determined by predistance polynomials and excesses.
\begin{thm}[Excess theorem (Fiol, Gago and Garriga~\cite{Fiol2010spo})]
Let $(X,R)$ be a connected regular graph with $d+1$ distinct eigenvalues
and diameter $D=d$,
$\{p_i\}^d_{i=0}$ be
the predistance polynomials,
and $\gamma _x$ be the excess of $x\in X$.
Then 
\[\
\frac{1}{|X|}\sum_{x\in X} \gamma _x \le p_d (\theta_0).
\]
Moreover equality is
attained if and only if $(X,R)$ is a distance-regular graph.
\end{thm}

There are many researches of distance-regularity
or walk-regularity
using predistance polynomials~\cite{Dalfo2012dco,
Dalfo2011oad,
Dalfo2008clm,
Dam2008set,
Dam2012asp,
Fiol1997fla,
Liu2011kwr}.

\begin{rem} \label{rem:drg2.4}
If $\Gamma$ is a distance-regular graph,
the predistance polynomials $\{p_i\}$ of $\Gamma$
coincide with the polynomial $\{v_i\}$ associated with the $P$-polynomial structure of $\Gamma$~\cite{Dam2008set}.
\end{rem}

The following lemma is used later.
\begin{lem}
\label{eq:graph_property}
Let $\Gamma$ be a connected regular graph
with the spectrum
$\{\theta_0 ^{m_0},\theta_1 ^{m_1},\ldots ,\theta_d ^{m_d}\}$
and the predistance polynomials $\{p_i\}^d_{i=0}$.
Then for each $h\in \{1,2,\ldots ,d\}$,
\[
\prod^{d}_{
\substack{
j=1\\
j\neq h
}}\frac{\theta_0-\theta _j}{\theta _h-\theta _j}
=
-\frac{m_h p_d(\theta _h)}{p_d(\theta_0)}.
\]
\end{lem}
\begin{proof}
Put $f_h(t)=\prod^{d}_{j=1, j\neq h}(t-\theta _j)
$,
for $i\in \{1,2,\ldots ,d\}$.
These polynomials have degree ${d-1}$ which is expressed as 
a linear combination of $\{p_i\}_{i=0}^{d-1}$, 
hence they are orthogonal to $p_d$.
Thus, note that $\Gamma$ is connected and $m_0=1$, and we have 
\[
0=
|X|\aprod{p_d}{f_h}
=
\sum^{d}_{i=0}
m_i p_d(\theta_i) f_h(\theta_i)
=
p_d(\theta_0) f_h(\theta_0)+m_h p_d(\theta _h) f_h(\theta _h). 
\]
Therefore we obtain the desired equation.
\end{proof}

\section{Proof of Theorem \ref{thm:main_thm}}
\label{sec:3}
First, we prove that (1) implies (2) in
Theorem \ref{thm:main_thm}.
Suppose $\mathfrak{X}$ is a $P$-polynomial association scheme
with respect to $A_1$.
Let $\{A_i\}^d_{i=0}$ be the $P$-polynomial ordering
and $\{v_i\}^d_{i=0}$ be the polynomials associated with the $P$-polynomial structure.
For each $h \in \{1,2\ldots ,d \}$,
we put
$\kappa_h=\prod^{d}_{j=1,\  j\neq h}
(\theta_0-\theta_j)/(\theta_{h}-\theta_{j})$,
and define the matrix   
\[
M^{\ast}_h:=\prod^{d}_{\substack{j=1\\
j \neq h}}
\frac{A_1- \theta_j I}{\theta_h-\theta_j}
\]
in $\mathfrak{A}$.
Since $M^{\ast}_h$ is a polynomial in $A_1$ of degree $d-1$,
it can be expressed as a combination of $\{A_i\}^{d-1}_{i=0}$.
Thus we have
\begin{equation}
\label{eq:M_ioA_s=0}
M^{\ast}_h \circ A_d = 0,
\end{equation}
where $\circ$ denotes the Hadamard product,
that is,
the entry-wise matrix product.
On the other hand,
we consider the expansion of $M^{\ast}_h$
in terms of the primitive idempotents $\{E_j \}^{d}_{j=0}$.
Since $A_1 E_j = \theta_j E_j$, we have
\[
M^{\ast}_h E_j =
\begin{cases}
\kappa_h E_0 & \text{if $j=0$,}\\
E_h & \text{if $j=h$,}\\
0 & \text{otherwise.}
\end{cases}
\]
This means
$M^{\ast}_h = \kappa_h E_0 +E_h$.
Hence,
by (\ref{eq:M_ioA_s=0}),
it follows that
$E_ h\circ A_d= -\kappa_h A_d/|X|$,
that is,
$Q_h(d)=-\kappa_h$.
Therefore the desired result follows.

Conversely,
we prove that (2) implies (1) in Theorem \ref{thm:main_thm}.
The following lemmas are used later.
\begin{lem}[\cite{Kurihara2012coq}]
\label{lem:sum_Ki_2}
For mutually distinct $\beta_1,\beta_2, \ldots, \beta_d$, 
the following formal identity holds:  
\[
\sum_{i=1}^d \beta_i^k
\prod_{\substack{j=1\\ j \ne  i}}^d
\frac{x-\beta_j}{\beta_i-\beta_j}=x^k\]
for all $k \in \{0,1, \ldots ,d-1\}$.
\end{lem}

We say that \emph{$A_j$ appears in an element $M$
of $\mathfrak{A}$}
if there exists a non-zero complex number $\alpha$
such that $M \circ A_j = \alpha A_j$. 
Let $N^{\ast}_k$ denote the set of indices $j$
such that $A_j$ appears in $A_1^{k}$ 
but does not appear in $A_1^{l}$ for each $0 \leq l \leq k-1$.
We remark that $N^{\ast}_0=\{0\}$ and $N^{\ast}_1=\{1\}$.
Moreover,
if $\theta_0$ is distinct from $\{\theta_j\}_{j=1}^d$,
then $J=|X| \prod^{d}_{j=1}\frac{A_1-\theta_j I}{\theta_0-\theta _j}$ holds by the property of the Hoffman polynomial. 
Therefore 
$\{0,1,\ldots ,d\}=\bigcup ^{d}_{k=0}N^{\ast}_k$
(a disjoint union).

\begin{lem}
\label{lem:key}
Suppose $\mathfrak{X}$ is a symmetric association scheme
of class $d$ satisfying $\theta_0$ is distinct from
$\{\theta_i\}_{i=1}^d$.
Then, the following are equivalent.
\begin{enumerate}[{$($}1{$)$}]
\item $\mathfrak{X}$ is a $P$-polynomial scheme
with respect to $A_1$.
\item $\bigcup ^{d-1}_{k=0} N^{\ast}_k
\neq
\{0,1,\ldots ,d\}$.
\end{enumerate}
\end{lem}

\begin{proof}
Suppose $\mathfrak{X}$ is a $P$-polynomial scheme
with respect to $A_1$.
Let $\{A_i\}^d_{i=0}$ be the $P$-polynomial ordering.
Then there exist polynomials $v_i$ of degree $i$,
such that $v_i( A_1)=A_i$ for any $i$. 
This implies that $N^{\ast}_i=\{i\}$
for any $i$,
that is,
$\bigcup ^{d-1}_{k=0} N^{\ast}_k
=
\{0,1,\ldots ,d-1\}$.
	 
Suppose $\bigcup ^{d-1}_{k=0} N^{\ast}_k
\neq
\{0,1,\ldots ,d\}$ holds.
Since $\bigcup ^{d}_{k=0} N^{\ast}_k
=
\{0,1,\ldots ,d\}$,
it follows $N^{\ast}_{d}$ is not empty.
If $N^{\ast}_{i}$ is empty for some $i$,
then $N^{\ast}_{i+1}$ is also empty.
This means that if $N^{\ast}_d$ is not empty,
then $N^{\ast}_{i}$ is not empty for any $i$.
Since $N^{\ast}_{0},N^{\ast}_{1},\ldots ,N^{\ast}_{d}$ are disjoint,
$N^{\ast}_{i}$ has size 1 for each $i\in \{0,1,\ldots ,d\}$.
Put $N^{\ast}_i=\{\xi_i\}$
and order $A_{\xi_0},A_{\xi_1},\ldots, A_{\xi_d}$.
Then we can construct polynomials $v_{i}$ of degree $i$
such that $v_{i}(A_1)=A_{\xi_i}$, and 
the first statement follows.
\end{proof}

\begin{rem}
From the proof of Lemma~\ref{lem:key},
an ordering of a $P$-polynomial association scheme with respect to
$A_1$ is uniquely determined.
\end{rem}

Let us return to prove that (2) implies (1)
in Theorem \ref{thm:main_thm}.
We have 
\[
A_1^k= \frac{\theta_0^k}{|X|}J + \sum_{h=1}^d \theta_h^k E_h.
\]
By our assumption,
it follows that
\[
E_h \circ A_l
=
\frac{Q_h (l)}{|X|}A_l\\
=
-\frac{1}{|X|}
\prod^{d}_{\substack{j=1\\
j \neq h}} \frac{\theta_0-\theta_j}{\theta_h-\theta_j}A_l.
\]
By Lemma \ref{lem:sum_Ki_2}, 
\[
A_1^k \circ A_l
=
\left(
\frac{\theta_0^k}{|X|}J + \sum_{h=1}^d \theta_h^k E_h
\right)
\circ A_l
=
\frac{1}{|X|}
\Biggl(
\theta_0^k-\sum_{h=1}^d \theta_h^k
\prod^{d}_{\substack{j=1\\
j \neq h}} \frac{\theta_0-\theta_j}{\theta_h-\theta_j}
\Biggr)
A_l=0
\]
for every $k \leq d-1$.
This means that $l$ is not an element of $N^{\ast}_k$
for every $k \leq d-1$.
By Lemma \ref{lem:key},
there exists an ordering $\{A_{\xi_i}\}^d_{i=0}$
such that $\mathfrak{X}$ is a $P$-polynomial scheme
with respect to $\{A_{\xi_i}\}^d_{i=0}$.
Moreover we have $\xi_{d}=l$.
Therefore the desired result follows.

\section{Proof of Theorem~\ref{thm:predistance-main}}
\label{sec:5}

First, we prove that (1) implies (2) in
Theorem \ref{thm:predistance-main}.
Suppose $\mathfrak{X}$ is a $P$-polynomial association scheme
with respect to $A_1$. Let $\{v_i\}^d_{i=0}$ be the polynomials
associated with the $P$-polynomial structure of $\mathfrak{X}$.
Then $\{v_i\}^d_{i=0}$ coincide with
the predistance polynomials $\{p_i\}^d_{i=0}$ of the graph $(X,R_1)$ by Remarks~\ref{rem:p-poly-drg} and \ref{rem:drg2.4}.
Therefore we have $p_d(\theta_h)=v_d(\theta_h)=P_d(h)$
for each $h\in \{0,1,\ldots,d\}$. 

Next, we prove that (2) implies (1) in
Theorem \ref{thm:predistance-main}.
Suppose
$p_d(\theta_h)=P_l(h)$
holds
for each $h\in \{0,1,\ldots ,d\}$.
Then by Lemma~\ref{eq:graph_property} and $k_l Q_h(l)=m_h P_l(h)$
\cite[Theorem~3.5, page 62]{Bannai1984aci} we have
\[
 \prod^{d}_{
\substack{
j=1\\
j\neq h
}}\frac{\theta_0-\theta _j}{\theta _h-\theta _j}
=
-\frac{m_h p_d(\theta _h)}{p_d(\theta_0)}
=-\frac{m_h P_l(h)}{k_l}= -Q_h(l)
\]
for each $h\in \{1,2,\ldots ,d\}$.
By Theorem~\ref{thm:main_thm}, $\mathfrak{X}$ is a
$P$-polynomial association scheme
with respect to $A_1$.

\bigskip

\noindent
\textbf{Acknowledgments.}
The authors thank  Edwin van Dam  for pointing out that the value in
Theorem~\ref{thm:main_thm}
can be expressed by predistance polynomials,
and the referees for some useful suggestions that improved the presentation.
The authors are also grateful to
Kazumasa Nomura, Hajime Tanaka and Paul Terwilliger
for careful reading and insightful comments.
The second author was supported by JSPS KAKENHI
Grant Number 25800011.

\end{document}